\documentclass[11pt]{article}

\usepackage[a4paper,margin=1in]{geometry}
\usepackage{amsmath,amssymb,amsthm}
\usepackage{authblk}
\usepackage{hyperref}
\hypersetup{colorlinks=true, linkcolor=black, citecolor=blue, urlcolor=blue}

\newtheorem{theorem}{Theorem}[section]
\newtheorem{lemma}[theorem]{Lemma}
\newtheorem{proposition}[theorem]{Proposition}
\newtheorem{remark}[theorem]{Remark}

\newcommand{\Ric}{\mathrm{Ric}}
\newcommand{\E}{\mathbb{E}}
\newcommand{\Var}{\mathrm{Var}}
\newcommand{\dd}{\,\mathrm{d}}
\newcommand{\de}{\tilde d}

\title{Variance-Refined In-Diameter Lower Bound for the First Dirichlet Eigenvalue}
\author{Thomas Sch\"urmann}
\affil{D\"usseldorf, Germany}
\date{}

\begin{document}
\maketitle

\begin{abstract}
	Let $(M,g)$ be a compact $n$-dimensional Riemannian manifold with nonempty boundary and $n\geq 2$. Assume that ${\Ric(M)\ge (n-1)K}$ for some ${K>0}$ and that $\partial M$ has nonnegative mean curvature with respect to the outward unit normal.
	Denote by $\lambda$ the first Dirichlet eigenvalue of the Laplacian.
	Ling’s gradient-comparison method \cite{Ling2006} provides an explicit lower bound for $\lambda$ in terms of $K$ and the in-diameter $\de$ (twice the maximal distance from a point of $M$ to $\partial M$).
	We isolate the only step in Ling’s argument that loses quantitative information: a Jensen-H\"older averaging that replaces a nonconstant one-dimensional comparison function by its mean.
	Using the uniform strong convexity of $x\mapsto x^{-1/2}$ on $(0,1]$, we refine this averaging by a variance term and thereby retain part of the discarded oscillation.
	This yields an explicit closed-form in-diameter bound that is strictly stronger than Ling’s estimate for every $K>0$.
\end{abstract}

\section{Introduction and main result}

Let $(M,g)$ be a compact $n$-dimensional Riemannian manifold with nonempty boundary $\partial M$.
Assume that the Ricci curvature satisfies
\begin{equation}
\label{eq:Ric-lower}
\Ric(M)\ge (n-1)K
\end{equation}
for some constant $K>0$, and that the mean curvature of $\partial M$ with respect to the outward unit normal is nonnegative.
Let $\lambda$ denote the first Dirichlet eigenvalue of the Laplacian on $M$.

A classical result of Reilly~\cite{Reilly1977} yields the Lichnerowicz-type estimate
\begin{equation}
\lambda \ge nK.
\end{equation}
This estimate contains no diameter information and becomes trivial in the limiting case $K=0$.
For closed manifolds, the case $K=0$ in \eqref{eq:Ric-lower} corresponds to nonnegative Ricci curvature; in that setting Li--Yau~\cite{LiYau1980} and Zhong--Yang~\cite{ZhongYang1984} obtained sharp diameter-type lower bounds.

In \cite{Ling2006}, Ling proved a unified in-diameter estimate.
Writing
\begin{equation}
\de = 2\sup_{x\in M}\mathrm{dist}(x,\partial M),
\end{equation}
Ling's main theorem yields
\begin{equation}
\lambda \ge \frac{(n-1)K}{2}+\frac{\pi^2}{\de^{\,2}}.
\end{equation}
The proof is based on a refined gradient comparison and introduces an auxiliary function $\xi$ on $[-\pi/2,\pi/2]$.
After reducing the eigenvalue estimate to a one-dimensional integral inequality, Ling applies H\"older's inequality to replace the nonconstant comparison function $z(t)$ by its mean.
This step discards information about the oscillation of $z$.

The aim of this note is to retain this oscillation quantitatively via a variance refinement.
We obtain an explicit strengthening of Ling's bound in closed form.

\begin{theorem}[Variance-refined in-diameter bound]
\label{thm:main}
Let $(M,g)$ satisfy the assumptions above and let $\lambda$ be the first Dirichlet eigenvalue.
Set
\begin{equation}
\alpha = \frac{(n-1)K}{2},
\qquad
D=\frac{\pi^2}{\de^{\,2}},
\qquad
V = 4\zeta(3)-\frac{1}{3}(\pi^2+4).
\end{equation}
Then
\begin{equation}
\label{eq:mainbound}
\lambda\ \ge\
\frac{(\alpha+D)+\sqrt{(\alpha+D)^2+V\,\alpha^2}}{2}.
\end{equation}
In particular, since $V=\Var(\xi)>0$ (see Remark~\ref{rem:V-positive}), one has the strict improvement
\begin{equation}
\lambda\ >\ \alpha + D
\qquad
\text{whenever }K>0.
\end{equation}
\end{theorem}

\begin{remark}
The bound \eqref{eq:mainbound} is derived from the same one-dimensional comparison inequality as in \cite{Ling2006} and is obtained in closed form, without taking a maximum with the estimate $\lambda\ge nK$.
We use Reilly's estimate $\lambda\ge nK$ only to ensure that the parameter $\delta=\alpha/\lambda$ lies in $[0,1/2)$.
\end{remark}

\section{The comparison inequality from Ling's argument}

We briefly recall the one-dimensional inequality that concludes Ling's proof of Theorem~1.
The full gradient comparison argument can be found in \cite{Ling2006}; for our purposes we only need the resulting integral inequality and the explicit auxiliary function $\xi$.

\begin{lemma}[The auxiliary function $\xi$]
\label{lem:xi}
Define $\xi:[-\pi/2,\pi/2]\to\mathbb{R}$ by
\begin{equation}
\label{eq:xi-def}
\xi(t)=\frac{\cos^2 t + 2t\sin t\cos t + t^2 - \pi^2/4}{\cos^2 t}.
\end{equation}
Then $\xi$ is smooth and even on $[-\pi/2,\pi/2]$, satisfies $\xi(\pm \pi/2)=0$, and
\begin{equation}
\label{eq:xi-mean}
\int_{0}^{\pi/2}\xi(t)\,\dd t = -\frac{\pi}{2}.
\end{equation}
Moreover $\xi(t)\le 0$ for $t\in[0,\pi/2]$.
\end{lemma}

\begin{proof}
These properties are established in Lemma~5 of \cite{Ling2006}, where $\xi$ is constructed explicitly and shown to satisfy a linear ODE.
For completeness, we derive \eqref{eq:xi-mean} from the first-order relation
\begin{equation}
\label{eq:xi-firstorder}
\bigl(\xi(t)\cos^2 t\bigr)' = 4t\cos^2 t
\end{equation}
together with the boundary value $\xi(\pi/2)=0$.
Integrating \eqref{eq:xi-firstorder} from $t$ to $\pi/2$ gives
\[
-\xi(t)\cos^2 t=\int_t^{\pi/2}4s\cos^2 s\,\dd s,
\qquad\text{hence}\qquad
\xi(t)=-4\sec^2 t\int_t^{\pi/2}s\cos^2 s\,\dd s.
\]
Using Fubini's theorem, we compute
\begin{align*}
\int_{0}^{\pi/2}\xi(t)\,\dd t
&=-4\int_{0}^{\pi/2}\sec^2 t\left(\int_{t}^{\pi/2}s\cos^2 s\,\dd s\right)\dd t\\
&=-4\int_{0}^{\pi/2}s\cos^2 s\left(\int_{0}^{s}\sec^2 t\,\dd t\right)\dd s
=-4\int_{0}^{\pi/2}s\cos^2 s\,\tan s\,\dd s\\
&=-4\int_{0}^{\pi/2}s\sin s\cos s\,\dd s
=-2\int_{0}^{\pi/2}s\sin(2s)\,\dd s
=-\frac{\pi}{2},
\end{align*}
which proves \eqref{eq:xi-mean}.
\end{proof}

\begin{lemma}[Ling's integral inequality]
\label{lem:ling-integral}
Let $\lambda$ be the first Dirichlet eigenvalue, set $\alpha=(n-1)K/2$ and $\delta=\alpha/\lambda$.
Define
\begin{equation}
\label{eq:z-def}
z(t)=1+\delta\,\xi(t),
\qquad t\in[0,\pi/2].
\end{equation}
Then
\begin{equation}
\label{eq:ling-star}
\sqrt{\lambda}\,\frac{\de}{2}\ \ge\ \int_{0}^{\pi/2}\frac{\dd t}{\sqrt{z(t)}}.
\end{equation}
\end{lemma}

\begin{proof}
This is \cite[(44)]{Ling2006}, obtained by integrating a gradient comparison inequality along a minimizing geodesic from a maximum point of the first Dirichlet eigenfunction to the boundary and then letting the normalization parameter $b\downarrow 1$.
By definition of the in-diameter, the length of such a geodesic is at most $\de/2$.
The function $z$ is the explicit comparison function used in \cite[(35)]{Ling2006}.
\end{proof}

\begin{remark}
Ling derives his explicit bound by applying H\"older's inequality (equivalently, Jensen's inequality for the convex function $x\mapsto x^{-1/2}$ with respect to the normalized measure $\tfrac{2}{\pi}\,\dd t$) to \eqref{eq:ling-star}:
\begin{equation}
\int_{0}^{\pi/2}\frac{\dd t}{\sqrt{z(t)}}
\ \ge\
\frac{\left(\int_0^{\pi/2}\dd t\right)^{3/2}}{\left(\int_0^{\pi/2} z(t)\,\dd t\right)^{1/2}}
=
\frac{\pi}{2}\cdot\frac{1}{\sqrt{1-\delta}},
\end{equation}
because $\frac{2}{\pi}\int_0^{\pi/2} z(t)\,\dd t=1-\delta$ by \eqref{eq:xi-mean}.
This step ignores that $z$ is nonconstant. Equality in this Jensen/H\"older step would force $z$ to be constant almost everywhere; since $\xi$ is not constant, this cannot occur when $\delta>0$.
We replace it by a variance-sensitive estimate.
\end{remark}

\section{A strong-convexity refinement}

The key observation is that $x\mapsto x^{-1/2}$ is uniformly strongly convex on $(0,1]$.

\begin{proposition}[Variance improvement for $x^{-1/2}$]
\label{prop:variance}
Let $z:[0,\pi/2]\to(0,1]$ be measurable and set
\begin{equation}
\mu = \frac{2}{\pi}\int_{0}^{\pi/2} z(t)\,\dd t.
\end{equation}
Then
\begin{equation}
\label{eq:variance-ineq}
\frac{2}{\pi}\int_{0}^{\pi/2}\frac{\dd t}{\sqrt{z(t)}}
\ \ge\
\frac{1}{\sqrt{\mu}}+\frac{3}{8}\,\Var(z),
\end{equation}
where
\begin{equation}
\Var(z)=\frac{2}{\pi}\int_{0}^{\pi/2}\bigl(z(t)-\mu\bigr)^2\,\dd t.
\end{equation}
\end{proposition}

\begin{proof}
Let $f(x)=x^{-1/2}$ on $(0,1]$.
Then
\begin{equation}
f''(x)=\frac{3}{4}x^{-5/2}\ge \frac{3}{4}
\qquad \text{for all }x\in(0,1].
\end{equation}
Fix $\mu\in(0,1]$ and use the second-order Taylor expansion of $f$ at $\mu$ with integral remainder.
Using the lower bound on $f''$, we obtain for every $x\in(0,1]$:
\begin{equation}
f(x)\ \ge\ f(\mu)+f'(\mu)(x-\mu)+\frac{3}{8}(x-\mu)^2.
\end{equation}
Apply this pointwise with $x=z(t)$ and integrate over $t\in[0,\pi/2]$.
The linear term vanishes because $\mu$ is the mean of $z$:
\begin{equation}
\frac{2}{\pi}\int_{0}^{\pi/2}\bigl(z(t)-\mu\bigr)\,\dd t = 0.
\end{equation}
This yields \eqref{eq:variance-ineq}.
\end{proof}

We now apply Proposition~\ref{prop:variance} to the specific choice $z(t)=1+\delta\xi(t)$ from \eqref{eq:z-def}.
Note that $z(t)\le 1$ on $[0,\pi/2]$ since $\xi(t)\le 0$ by Lemma~\ref{lem:xi}.
Moreover, \eqref{eq:xi-firstorder} implies the pointwise lower bound $\xi(t)\ge -2$ on $[0,\pi/2]$: indeed, integrating \eqref{eq:xi-firstorder} from $t$ to $\pi/2$ gives
\[
-\xi(t)\cos^2 t\,=\,\int_t^{\pi/2}4s\cos^2 s\,\dd s,
\]
hence
\[
-\xi(t)=4\sec^2 t\int_t^{\pi/2}s\cos^2 s\,\dd s.
\]
Setting $I(t)=\int_t^{\pi/2}s\cos^2 s\,\dd s$ and $F(t)=\tfrac12\cos^2 t-I(t)$, we have $F(\pi/2)=0$ and
\[
F'(t)=-\cos t\sin t+t\cos^2 t=\cos^2 t\,(t-\tan t)\le 0,
\]
since $\tan t\ge t$ for $t\in[0,\pi/2)$.
Thus $F(t)\ge 0$ and $I(t)\le \tfrac12\cos^2 t$, which yields $-\xi(t)\le 2$.
Since Reilly's estimate gives $\delta=\alpha/\lambda\le (n-1)/(2n)<1/2$, we obtain $z(t)=1+\delta\xi(t)\ge 1-2\delta>0$.
Therefore $z(t)\in(0,1]$ and Proposition~\ref{prop:variance} applies.

\begin{lemma}[Mean and variance of $z(t)=1+\delta\xi(t)$]
\label{lem:meanvar-z}
Let $z(t)=1+\delta\xi(t)$ on $[0,\pi/2]$, where $\xi$ is given by \eqref{eq:xi-def}.
Then
\begin{equation}
\label{eq:mu-1-delta}
\mu = \frac{2}{\pi}\int_0^{\pi/2} z(t)\,\dd t = 1-\delta,
\end{equation}
and
\begin{equation}
\label{eq:var-z}
\Var(z)=\delta^2\,\Var(\xi),
\qquad
\Var(\xi)=\E[\xi^2]-1,
\qquad
\E[\xi^2]=\frac{2}{\pi}\int_0^{\pi/2}\xi(t)^2\,\dd t.
\end{equation}
\end{lemma}

\begin{proof}
The identity \eqref{eq:mu-1-delta} follows immediately from \eqref{eq:xi-mean}.
Since $\E[\xi]=\frac{2}{\pi}\int_0^{\pi/2}\xi(t)\,\dd t=-1$, we have
\begin{equation}
z(t)-\mu = 1+\delta\xi(t) - (1-\delta)=\delta(\xi(t)+1),
\end{equation}
and hence \eqref{eq:var-z}.
\end{proof}

Combining Proposition~\ref{prop:variance} and Lemma~\ref{lem:meanvar-z} yields the refined lower bound on the integral in \eqref{eq:ling-star}:

\begin{proposition}[Variance-refined integral estimate]
\label{prop:refined-integral}
Let $z(t)=1+\delta\xi(t)$ as in \eqref{eq:z-def}.
Then
\begin{equation}
\label{eq:integral-refined}
\int_{0}^{\pi/2}\frac{\dd t}{\sqrt{z(t)}}
\ \ge\
\frac{\pi}{2}\left(\frac{1}{\sqrt{1-\delta}}+\frac{3}{8}\Var(\xi)\,\delta^2\right).
\end{equation}
In particular, if $\Var(\xi)>0$ and $\delta>0$, the right-hand side is strictly larger than $\frac{\pi}{2}(1-\delta)^{-1/2}$.
\end{proposition}

\begin{proof}
This is \eqref{eq:variance-ineq} with $\mu=1-\delta$ and $\Var(z)=\delta^2\Var(\xi)$.
\end{proof}

\section{Evaluation of \texorpdfstring{$\Var(\xi)$}{Var(xi)}}

We now compute the constant $\Var(\xi)$ in closed form.
Set
\begin{equation}
V := \Var(\xi).
\end{equation}

\begin{lemma}[Second moment of $\xi$]
\label{lem:xi-secondmoment}
For $\xi$ defined by \eqref{eq:xi-def},
\begin{equation}
\label{eq:xi-square-integral}
\int_{0}^{\pi/2}\xi(t)^2\,\dd t
=
\pi\left(2\zeta(3)-\frac{\pi^2+1}{6}\right).
\end{equation}
Consequently,
\begin{equation}
\label{eq:V-closedform}
V = 4\,\zeta(3)-\frac{1}{3}(\pi^2+4).
\end{equation}
\end{lemma}

\begin{proof}
A detailed reduction of $\int_0^{\pi/2}\xi(t)^2\,\dd t$ to a short list of logarithmic integrals is given in Appendix~\ref{app:lem4.1-details}.
We record the remaining (standard) Fourier-series evaluations.
First, rewrite \eqref{eq:xi-def} as
\begin{equation}
\label{eq:xi-alt}
\xi(t) = 1 + 2t\tan t + \left(t^2-\frac{\pi^2}{4}\right)\sec^2 t.
\end{equation}
Introduce $L(t)=\log(\cos t)$ on $(0,\pi/2)$, so that $L'(t)=-\tan t$ and $L''(t)=-\sec^2 t$.
Then \eqref{eq:xi-alt} becomes
\begin{equation}
\xi(t)=1-2t\,L'(t)-\left(t^2-\frac{\pi^2}{4}\right)L''(t).
\end{equation}
Expanding $\xi(t)^2$ and integrating by parts repeatedly reduces $\int_0^{\pi/2}\xi(t)^2\,\dd t$ to a linear combination of the three classical integrals
\begin{equation}
\int_0^{\pi/2} L(t)\,\dd t,
\qquad
\int_0^{\pi/2} t\,L(t)\,\dd t,
\qquad
\int_0^{\pi/2} t^2\,L(t)\,\dd t,
\end{equation}
plus elementary polynomial integrals.
The singular boundary terms cancel because $\xi(\pi/2)=0$ and $\xi(t)\cos^2 t$ is smooth up to $t=\pi/2$.

It therefore suffices to evaluate the three logarithmic integrals above.
For $|t|<\pi/2$ one has the absolutely convergent Fourier series
\begin{equation}
\label{eq:logcos-fourier}
\log(2\cos t)=\sum_{k=1}^{\infty}(-1)^{k+1}\frac{\cos(2kt)}{k}.
\end{equation}
Integrating term-by-term yields
\begin{equation}
\int_0^{\pi/2} \log(\cos t)\,\dd t = -\frac{\pi}{2}\log 2.
\end{equation}
A second integration against $t$ and $t^2$, using
\begin{equation}
\int_0^{\pi/2} t\,\cos(2kt)\,\dd t = \frac{(-1)^k-1}{4k^2},
\qquad
\int_0^{\pi/2} t^2\,\cos(2kt)\,\dd t = \frac{\pi(-1)^k}{4k^2},
\end{equation}
shows that
\begin{equation}
\int_0^{\pi/2} t\,\log(\cos t)\,\dd t
=
-\frac{\pi^2}{8}\log 2-\frac{7}{16}\zeta(3),
\end{equation}
and
\begin{equation}
\int_0^{\pi/2} t^2\,\log(\cos t)\,\dd t
=
-\frac{\pi^3}{24}\log 2-\frac{\pi}{4}\zeta(3),
\end{equation}
where we use the classical identity for the alternating zeta value
\begin{equation}
\sum_{k=1}^{\infty}(-1)^{k+1}\frac{1}{k^3}=\left(1-2^{-2}\right)\zeta(3)=\frac{3}{4}\zeta(3).
\end{equation}
Substituting these evaluations into the reduction from Appendix~\ref{app:lem4.1-details} yields \eqref{eq:xi-square-integral}.
Finally, since $\E[\xi]=-1$, we have
\begin{equation}
V=\E[\xi^2]-\E[\xi]^2=\E[\xi^2]-1
=
\frac{2}{\pi}\int_0^{\pi/2}\xi(t)^2\,\dd t-1,
\end{equation}
which gives \eqref{eq:V-closedform}.
\end{proof}

\begin{remark}\label{rem:V-positive}
The constant $V$ is a variance, hence nonnegative by definition.
Moreover $V>0$ because $\xi$ is not (a.e.) constant: using \eqref{eq:xi-alt} one has
\(\xi(0)=1-\pi^2/4<0\) (since $\pi>2$), while $\xi(\pi/2)=0$ by Lemma~\ref{lem:xi}.
For reference, $V\approx 0.1850261456$.
\end{remark}

\section{From the refined integral to an explicit eigenvalue bound}

We now combine the refined integral estimate with a simple one-root majorization to obtain the explicit bound \eqref{eq:mainbound}.

\begin{lemma}[One-root lower bound]
\label{lem:one-root}
Let $V>0$ be as in \eqref{eq:V-closedform}.
For every $\delta\in[0,1/2]$ one has
\begin{equation}
\label{eq:one-root}
\frac{1}{\sqrt{1-\delta}}+\frac{3}{8}V\,\delta^2
\ \ge\
\frac{1}{\sqrt{1-\delta-\frac{V}{4}\delta^2}}.
\end{equation}
\end{lemma}

\begin{proof}
Fix $\delta\in[0,1/2]$ and consider the function
\begin{equation}
h(s)=(1-\delta-s\delta^2)^{-1/2}
\qquad\text{for } s\in\left[0,\frac{V}{4}\right].
\end{equation}
Then $h$ is increasing and convex in $s$.
By convexity,
\begin{equation}
h\left(\frac{V}{4}\right)-h(0)\le \frac{V}{4}\,h'\left(\frac{V}{4}\right).
\end{equation}
Since
\begin{equation}
h'(s)=\frac{\delta^2}{2}(1-\delta-s\delta^2)^{-3/2},
\end{equation}
we obtain
\begin{equation}
h\left(\frac{V}{4}\right)
\le
\frac{1}{\sqrt{1-\delta}}+\frac{V}{8}\,\delta^2\left(1-\delta-\frac{V}{4}\delta^2\right)^{-3/2}.
\end{equation}
For $\delta\in[0,1/2]$ we have the uniform lower bound
\begin{equation}
1-\delta-\frac{V}{4}\delta^2 \ge \frac{1}{2}-\frac{V}{16}.
\end{equation}
We now bound this factor uniformly using only explicit analytic inequalities (in particular, without inserting a decimal approximation for $V$). 

Recall from \eqref{eq:V-closedform} that $V=4\,\zeta(3)-\frac{1}{3}(\pi^2+4)$.
We first claim that $V<\frac{1}{4}$.
By the integral test,
\begin{equation}
\zeta(3)=\sum_{m=1}^{\infty}\frac{1}{m^3}
<
\sum_{m=1}^{5}\frac{1}{m^3}+\int_{5}^{\infty}x^{-3}\,\dd x
=
\left(1+\frac18+\frac1{27}+\frac1{64}+\frac1{125}\right)+\frac1{50}
=
\frac{260423}{216000}.
\end{equation}
Hence $4\zeta(3)<\frac{260423}{54000}$.
On the other hand, the classical bound $\pi>\frac{223}{71}$ implies
$\pi^2>\left(\frac{223}{71}\right)^2=\frac{49729}{5041}>\frac{493}{50}$, so
\begin{equation}
\frac{\pi^2+4}{3}>\frac{\frac{493}{50}+4}{3}=\frac{231}{50}.
\end{equation}
Consequently,
\begin{equation}
V = 4\zeta(3)-\frac{1}{3}(\pi^2+4)<\frac{260423}{54000}-\frac{231}{50}
=\frac{10943}{54000}<\frac14.
\end{equation}
Using $V<\frac14$ we get $\frac12-\frac{V}{16}>\frac12-\frac{1}{64}=\frac{31}{64}$, and therefore
\begin{equation}
\left(\frac12-\frac{V}{16}\right)^{-3/2}<
\left(\frac{31}{64}\right)^{-3/2}=
\left(\frac{64}{31}\right)^{3/2}<3,
\end{equation}
since
\(
\left(\frac{64}{31}\right)^{3/2}<3
\iff
\left(\frac{64}{31}\right)^3<9
\),
and indeed
\(
\left(\frac{64}{31}\right)^3=\frac{262144}{29791}<9
\)
because $9\cdot 29791=268119>262144$.
Therefore
\begin{equation}
\left(1-\delta-\frac{V}{4}\delta^2\right)^{-3/2}\le 3,
\end{equation}
and hence
\begin{equation}
h\left(\frac{V}{4}\right)
\le
\frac{1}{\sqrt{1-\delta}}+\frac{3V}{8}\,\delta^2,
\end{equation}
which is exactly \eqref{eq:one-root}.
\end{proof}

\begin{remark}
Lemma~\ref{lem:one-root} is used only to turn the additive variance correction in \eqref{eq:integral-refined} into the simple denominator in \eqref{eq:lambda-rational}, and hence into the closed-form bound \eqref{eq:mainbound}. If one keeps \eqref{eq:integral-refined} without Lemma~\ref{lem:one-root}, one gets a slightly stronger (but implicit) lower bound for $\lambda$.
\end{remark}

\begin{lemma}[Range of $\delta$]
\label{lem:delta-range}
Under the standing geometric assumptions,
\begin{equation}
0\le \delta=\frac{\alpha}{\lambda}\le \frac{n-1}{2n}<\frac{1}{2}.
\end{equation}
\end{lemma}

\begin{proof}
By Reilly's estimate $\lambda\ge nK$ and $\alpha=(n-1)K/2$ we obtain
\begin{equation}
\delta=\frac{\alpha}{\lambda}\le \frac{(n-1)K/2}{nK}=\frac{n-1}{2n}.
\end{equation}
\end{proof}

\begin{proof}[Proof of Theorem~\ref{thm:main}]
Set $\delta=\alpha/\lambda$.
By Lemma~\ref{lem:ling-integral}, Proposition~\ref{prop:refined-integral}, Lemma~\ref{lem:xi-secondmoment}, and Lemmas~\ref{lem:one-root}--\ref{lem:delta-range}, we have
\begin{equation}
\sqrt{\lambda}\,\frac{\de}{2}
\ \ge\
\int_{0}^{\pi/2}\frac{\dd t}{\sqrt{z(t)}}
\ \ge\
\frac{\pi}{2}\cdot\frac{1}{\sqrt{1-\delta-\frac{V}{4}\delta^2}}.
\end{equation}
Squaring and writing $D=\pi^2/\de^{\,2}$ yields
\begin{equation}
\label{eq:lambda-rational}
\lambda \ge \frac{D}{1-\delta-\frac{V}{4}\delta^2}.
\end{equation}
Substitute $\delta=\alpha/\lambda$ into \eqref{eq:lambda-rational} and clear denominators:
\begin{equation}
\lambda\left(1-\frac{\alpha}{\lambda}-\frac{V}{4}\frac{\alpha^2}{\lambda^2}\right)\ge D.
\end{equation}
Equivalently,
\begin{equation}
\lambda^2-(\alpha+D)\lambda-\frac{V}{4}\alpha^2\ge 0.
\end{equation}
Since $\lambda>0$, this quadratic inequality implies
\begin{equation}
\lambda\ge \frac{(\alpha+D)+\sqrt{(\alpha+D)^2+V\,\alpha^2}}{2},
\end{equation}
which is \eqref{eq:mainbound}.

Finally, $V>0$ by \eqref{eq:V-closedform}, so the square-root term is strictly larger than $\alpha+D$ whenever $\alpha>0$, proving the strict improvement when $K>0$.
\end{proof}

\section{Concluding remark}

In order to make the size of the refinement more transparent, let
\[
B_{\mathrm{Ling}}:=\alpha+D
\qquad\text{and}\qquad
B_{\mathrm{var}}:=\frac{(\alpha+D)+\sqrt{(\alpha+D)^2+V\,\alpha^2}}{2}
\]
denote, respectively, Ling's lower bound and the variance-refined bound \eqref{eq:mainbound}.
A direct computation yields
\[
\frac{B_{\mathrm{var}}}{B_{\mathrm{Ling}}}
=
\frac{1+\sqrt{1+V\left(\frac{\alpha}{\alpha+D}\right)^2}}{2},
\]
so in particular
\[
1\ <\ \frac{B_{\mathrm{var}}}{B_{\mathrm{Ling}}}
\ \le\
\frac{1+\sqrt{1+V}}{2}
\ \approx\ 1.0443.
\]
Thus the improvement over Ling's estimate is universally bounded by about $4.5\%$ in relative terms.
The gain is governed by the dimensionless ratio
\[
\frac{\alpha}{D}=\frac{(n-1)K\,\de^{\,2}}{2\pi^2},
\]
and becomes more pronounced when the curvature term $\alpha$ dominates the diameter term $D$.
As a model case, for a geodesic hemisphere of the round $n$--sphere scaled so that $\Ric=(n-1)K$,
one has $\de=\pi/\sqrt{K}$ and hence $D=K$, so that $\alpha/(\alpha+D)=(n-1)/(n+1)$; this yields a
relative gain in the lower bound of about $3\%$ already for $n=10$, and it approaches the universal
limit $(1+\sqrt{1+V})/2-1\approx 4.4\%$ as $n\to\infty$.\\

In summary the improvement \eqref{eq:mainbound} is small but uniform: it depends only on the universal constant $V=\Var(\xi)>0$ that measures the nonconstancy of the one-dimensional comparison function $z(t)$.
It shows that the H\"older/Jensen reduction in \cite{Ling2006} is not optimal and can be sharpened while retaining a closed-form dependence on $K$ and $\de$.

\appendix

\section{Details for Lemma~\texorpdfstring{\ref{lem:xi-secondmoment}}{4.1}}
\label{app:lem4.1-details}

Write $a:=\pi/2$ and $A:=\pi^2/4$, and set $P(t):=t^2-A$.
For $0<\varepsilon<a$ define $u:=a-\varepsilon$ and
\(
I_{\varepsilon}:=\int_{0}^{u}\xi(t)^2\,\dd t
\),
where $\xi$ is given by \eqref{eq:xi-alt}.
Since $u<a$, all functions below are smooth on $[0,u]$, so the integrations by parts are classical; we take the limit $\varepsilon\downarrow 0$ at the end.

\subsection*{A.1. Reduction to logarithmic integrals}
Expanding \eqref{eq:xi-alt} gives
\begin{equation}
\label{eq:appendix-expand}
\xi(t)^2
=
1+4t^2\tan^2 t + P(t)^2\sec^4 t +4t\tan t +2P(t)\sec^2 t +4tP(t)\tan t\sec^2 t.
\end{equation}
Using $\tan^2 t=\sec^2 t-1$, we rewrite the first two terms on the right as
\(
1+4t^2\tan^2 t = 1-4t^2+4t^2\sec^2 t
\)
and hence
\begin{equation}
\label{eq:appendix-group}
\xi(t)^2
=
1-4t^2 + 4t\tan t + 2(3t^2-A)\sec^2 t + P(t)^2\sec^4 t + 4tP(t)\tan t\sec^2 t.
\end{equation}

We now integrate each group in \eqref{eq:appendix-group} over $[0,u]$.
First, integrating the $\sec^2$-term by parts using $(\tan t)'=\sec^2 t$ yields
\begin{equation}
\label{eq:appendix-sec2-first}
\int_0^u 2(3t^2-A)\sec^2 t\,\dd t
= 2(3u^2-A)\tan u - 12\int_0^u t\tan t\,\dd t.
\end{equation}
Second, we treat the $\sec^4$-term using the elementary identity
\begin{equation}
\label{eq:appendix-sec4-id}
\bigl(\tan t\,\sec^2 t\bigr)' = 3\sec^4 t - 2\sec^2 t,
\end{equation}
which follows by direct differentiation.
Multiplying \eqref{eq:appendix-sec4-id} by $P(t)^2$ and integrating by parts gives
\begin{align}
\label{eq:appendix-sec4-reduce}
\int_0^u P(t)^2\sec^4 t\,\dd t
&= \frac13 P(u)^2\tan u\,\sec^2 u
 - \frac13\int_0^u (P(t)^2)'\tan t\,\sec^2 t\,\dd t
 + \frac23\int_0^u P(t)^2\sec^2 t\,\dd t.
\end{align}
Since $(P(t)^2)'=4tP(t)$, the middle integral in \eqref{eq:appendix-sec4-reduce} combines with the last term in \eqref{eq:appendix-group}.
Using moreover
\begin{equation}
\label{eq:appendix-tansec2}
\tan t\,\sec^2 t = \frac12\bigl(\sec^2 t\bigr)'
\end{equation}
and integrating by parts once more, one arrives at the identity
\begin{align}
\label{eq:appendix-Ieps-formula}
I_{\varepsilon}
&=
u-\frac{4}{3}u^3
 + B_{\varepsilon}
 + \frac{2\pi^2}{3}\int_0^u \log(\cos t)\,\dd t
 - 8\int_0^u t^2\log(\cos t)\,\dd t,
\end{align}
where the boundary term is
\begin{align}
\label{eq:appendix-boundary}
B_{\varepsilon}
:=\;&
\frac{2}{3}(3u^2-A)\tan u
+\frac{1}{3}P(u)^2\tan u\,\sec^2 u\nonumber\\
+&\frac{4}{3}uP(u)\sec^2 u
+\frac{2}{3}P(u)^2\tan u
+\frac{8}{3}(u^3-Au)\log(\cos u).
\end{align}
The algebraic manipulations leading to \eqref{eq:appendix-Ieps-formula} use only \eqref{eq:appendix-sec2-first}--\eqref{eq:appendix-tansec2} and repeated integration by parts.

\subsection*{A.2. Passage to the endpoint}
As $\varepsilon\downarrow 0$ one has $u\uparrow a$ and
\(
\tan u = \cot\varepsilon,\ \sec^2 u = \csc^2\varepsilon,\ \log(\cos u)=\log(\sin\varepsilon).
\)
Using $P(u)=u^2-A= -\pi\varepsilon+O(\varepsilon^2)$ and the standard expansions
\begin{equation}
	\cot\varepsilon = \frac{1}{\varepsilon}-\frac{\varepsilon}{3}+O(\varepsilon^3),
	\qquad
	\csc^2\varepsilon = \frac{1}{\varepsilon^2}+\frac{1}{3}+O(\varepsilon^2),
	\qquad
	\log(\sin\varepsilon)=\log\varepsilon+O(\varepsilon^2),
\end{equation}
at $0$, one checks that all singular contributions in \eqref{eq:appendix-boundary} cancel and
\begin{equation}
	\label{eq:appendix-boundary-limit}
	\lim_{\varepsilon\downarrow 0} B_{\varepsilon} = -\frac{2\pi}{3}.
\end{equation}
Letting $\varepsilon\downarrow 0$ in \eqref{eq:appendix-Ieps-formula} therefore yields the reduction
\begin{equation}
	\label{eq:appendix-reduction-final}
	\int_0^{\pi/2}\xi(t)^2\,\dd t
	= \frac{\pi}{2}-\frac{\pi^3}{6}-\frac{2\pi}{3}
	+ \frac{2\pi^2}{3}\int_0^{\pi/2}\log(\cos t)\,\dd t
	- 8\int_0^{\pi/2}t^2\log(\cos t)\,\dd t.
\end{equation}
Substituting the two logarithmic integrals evaluated in the proof of Lemma~\ref{lem:xi-secondmoment} into \eqref{eq:appendix-reduction-final} gives \eqref{eq:xi-square-integral}.

\end{document}